\documentclass[sn-mathphys,Numbered]{sn-jnl}

\jyear{2023}%

\theoremstyle{thmstyleone}%
\newtheorem{theorem}{Theorem}
\newtheorem{proposition}[theorem]{Proposition}%

\theoremstyle{thmstyletwo}%
\newtheorem{example}{Example}%
\newtheorem{remark}{Remark}%

\newtheorem{definition}{Definition}%

\theoremstyle{thmstylethree}%
\newtheorem{corollary}{Corollary}%

\theoremstyle{thmstylethree}%
\newtheorem{lemma}{Lemma}%

\newtheorem{abctheorem}{Theorem}

\usepackage{color}
\usepackage{enumerate}
\usepackage{multirow}

\newcommand{\bfe}{\mathbf{e}}
\newcommand{\rmd}{\mathrm{d}}

\newcommand{\snabla}{\nabla^{\Sigma}}

\begin{document}

\title[Curves and surfaces of constant angle]{Curves and surfaces making a constant angle with a parallel transported direction in Riemannian spaces}


\author[1,2]{\fnm{Luiz C. B.} \sur{da Silva}}\email{LBarbosaDaSilva001@dundee.ac.uk}

\author[3]{\fnm{Gilson S.} \sur{Ferreira Jr}}\email{gilson.simoesj@ufrpe.br}

\author[3]{\fnm{Jos\'e D.} \sur{da Silva}}\email{jose.dsilva@ufrpe.br}

\affil[1]{\orgdiv{Division of Mathematics, School of Science and Engineering}, \orgname{University of Dundee}, \orgaddress{\city{Dundee}, \postcode{DD1 4HN}, \country{UK}}}

\affil[2]{\orgdiv{Department of Physics of Complex Systems}, \orgname{Weizmann Institute of Science}, \orgaddress{\city{Rehovot}, \postcode{7610001}, \country{Israel}}}

\affil[3]{\orgdiv{Departamento de Matem\'atica}, \orgname{Universidade Federal Rural de Pernambuco}, \orgaddress{\city{Recife}, \postcode{52171-900}, \state{PE}, \country{Brazil}}}



\abstract{In the last two decades, much effort has been dedicated to studying curves and surfaces according to their angle with a given direction. However, most findings were obtained using a case-by-case approach, and it is often unclear what are consequences of the specificities of the ambient manifold and what could be generic. In this work, we propose a theoretical framework to unify parts of these findings. We study curves and surfaces by prescribing the angle they make with a parallel transported vector field. We show that the characterization of Euclidean helices in terms of their curvature and torsion is also valid in any Riemannian manifold. Among other properties, we prove that surfaces making a constant angle with a parallel transported direction are extrinsically flat ruled surfaces. We also investigate the relation between their geodesics and the so-called slant helices; we prove that surfaces of constant angle are the rectifying surface of a slant helix, i.e., the ruled surface with rulings given by the Darboux vector field of the directrix. We characterize rectifying surfaces of constant angle; in other words, when their geodesics are slant helices. As a corollary, we show that if every geodesic of a surface of constant angle is a slant helix, then the ambient manifold is flat. Finally, we characterize surfaces in the product of a Riemannian surface with the real line making a constant angle with the vertical real direction.
}




\maketitle

\section{Introduction}
\label{sec::Intro}

A major goal of any geometer is to classify geometric objects according to some notion of simplicity. Thus placing the studied objects within a hierarchy. For example, one may study surfaces with zero curvature or Riemannian manifolds with surfaces displaying a certain behavior. In this work, we will be interested in defining classes of curves and surfaces by prescribing the angle they make with a given direction. Naturally, such a classification will greatly depend on the properties of the chosen direction, and one may expect that the simplest curves and surfaces are those making a constant angle with a ``fixed direction".


What a ``fixed direction" means in Euclidean space is quite intuitive, but this is not the case in other spaces. Since most manifolds are not isotropic, a privileged direction $\vec{v}$ in $M^{3}$ exists. Such direction makes defining submanifolds in $M^{3}$ of constant angle easy. In many cases, the direction $\Vec{v}$ is a Killing vector field, as illustrated by several results in  $\mathbb{S}^{2}\times\mathbb{R}$ \cite{dillen2007constantangleS2xR}, in $\mathbb{H}^{2}\times\mathbb{R}$ \cite{dillen2009constantangleH2xR}, in $H_{3}$  \cite{FastenakelsAMS2011}, in $Sol_{3}$ \cite{LopezKJM2011}, in $SL(2,\mathbb{R})$ \cite{MontaldoIJM2014}, in the warped product $I\times_f\mathbb{E}^2$ \cite{dillen2011constantanglewarpedproduct}, 
and in the Berger sphere $\mathbb{S}^{3}_{\epsilon}$ \cite{MendoncaJG2019}. (Most of these examples were obtained in homogeneous manifolds. Except for $\mathbb{S}^2\times\mathbb{R}$, all homogeneous Riemannian manifolds can be seen as a metric Lie group with a semidirect product $\mathbb{R}^2\rtimes\mathbb{R}$ \cite{Meeks2012metricLiegroups,Meeks-Ramos31CBM}. Thus, a natural class of curves and surfaces of constant angle could be obtained by measuring the angle with the vertical direction.) 

Riemannian manifolds generally fail to have symmetries, thus making a theory of constant angle based on Killing fields of limited applicability. To overcome the difficulties associated with the use of Killing fields, in this work, we propose to employ angles with respect to parallel transported fields. In other words, we study curves and surfaces on a three-dimensional Riemannian manifold $M^3$ that make a constant angle with a vector field $V$ parallel transported along the curve and the surface. Intuitively, we have two objects of distinct types, a curve (or surface) and a vector $V$, and we want to perform an angular measurement between them. Using parallel transport, we rigidly move $V$ along the curve (or surface) and then perform the measurement. This reasoning suggests that we should obtain a quite flexible theory for curves of constant angle as parallel transport always exists along curves regardless of the ambient space. On the other hand, the theory for surfaces may be less flexible as parallel transport along any path on a surface depends on the ambient manifold and also on the extrinsic properties of the surface. Our results will confirm this intuition.

The idea of using parallel transported vector fields as a fixed direction to study curves in a Riemannian manifold can be found in a paper by Hayden from the 1930s \cite{Hayden1931helicesRiemannianManifolds}, where the so-called generalized helices are characterized. To our knowledge, the same idea has not been applied to study surfaces or other helices.

In this work, we extend to a generic Riemannian manifold the study of generalized and slant helices with a parallel transported axis (Sect. \ref{sect::CurvesConstAngle}). Sect. \ref{Sect::HelicesAndCylinders} investigates the relation between generalized helices and geodesics of cylinders. We also study surfaces of constant angle (Sect. \ref{sect::SurfContAngle}) and their relation to slant helices (Sect. \ref{sect::SlantHelicesAndSurfConstAngle}). Finally, in Sect. \ref{sect::Conclusion}, we present our concluding remarks.

We finish this introductory section by giving a summary of our main results. (The theorems below are not in 1-1 correspondence with the results in the rest of the text.) In Section \ref{sect::CurvesConstAngle}, we establish a Lancret Theorem for helices on a generic manifold: 

\begin{abctheorem}\label{thrABC::GeneralLancretThr}
    A curve $\gamma: I\rightarrow M^{3}$ with curvature $\kappa>0$ makes a constant angle with a parallel transported field $V$ if, and only if, the ratio $\tau / \kappa$ of torsion to curvature is constant. In addition, $\tau / \kappa=\cot(\theta)$ and  $V=\cos(\theta) T+\sin(\theta) B$, where $\{T,N,B\}$ is the Frenet frame of $\gamma$.
\end{abctheorem}

Theorem \ref{thrABC::GeneralLancretThr} extends the Lancret Theorem in Euclidean space, see, e.g., Theorem 15.1 of Ref. \cite{Kreyszig1991}, and is a restriction to dimension 3 of Hayden's study \cite{Hayden1931helicesRiemannianManifolds}. It also extends results in $\mathbb{S}^2\times\mathbb{R}$, Theorem 1 of Ref. \cite{nistor2017constantangleS2xR}. (Note that the vertical real direction $\partial_t$ in $\mathbb{S}^2\times\mathbb{R}$ is parallel transported along any curve.)

 In Section \ref{sect::CurvesConstAngle}, we also provide a characterization of slant helices: 

\begin{abctheorem}\label{thrABC::SlantHelix}
    A curve $\gamma: I\rightarrow M^{3}$ with $\kappa>0$ is a slant  helix if, and only if, 
    \[
      \frac{\rmd\sigma}{\rmd s}=0,\quad \sigma = \frac{\kappa^2}{(\kappa^2+\tau^2)^{\frac{3}{2}}}\frac{\rmd }{\rmd s}\left(\frac{\tau}{\kappa}\right).
    \]
    In addition, $\sigma=\cot(\theta)$ and  $V=\cos(\theta) N+\sin(\theta) D$, where $D=(\tau T+\kappa B)/\sqrt{\tau^2+\kappa^2}$.
\end{abctheorem}

Theorem \ref{thrABC::SlantHelix} extends the study of slant helices in Euclidean space \cite{IzumiyaTJM2004}. To our knowledge, Theorem \ref{thrABC::SlantHelix} is the first result concerning slant helices on a generic manifold with respect to a parallel transported direction. We also show how to define a normal indicatrix of a curve $\gamma:I\to M^3$, i.e., how to see the principal normal as a spherical curve, which allows us to geometrically interpret $\sigma$; it is the geodesic curvature of the normal indicatrix. Constructing a tangent indicatrix also allows us to provide an interpretation for the ratio of torsion to curvature in Theorem \ref{thrABC::GeneralLancretThr} as a spherical curvature.

It is known that generalized helices in Euclidean space are characterized as geodesics of cylinders \cite{Kreyszig1991}. In Section \ref{Sect::HelicesAndCylinders}, we investigate the relation between generalized helices and geodesics of intrinsic cylinders. By an intrinsic cylinder, we mean the ruled surface whose director field of the rulings is parallel transported along the generating curve. (Hayden did not investigate whether generalized helices are cylindrical geodesics in Ref. \cite{Hayden1931helicesRiemannianManifolds}.) The main result of Section \ref{Sect::HelicesAndCylinders} is

\begin{abctheorem}
    Every geodesic of an intrinsic cylinder $\mathcal{C}\subset M^3$ is a generalized helix if, and only if, $\mathcal{C}$ is intrinsically and extrinsically flat. In addition, $M^3$ is flat if, and only if, every geodesic of any intrinsic cylinder is a generalized helix.
\end{abctheorem}

In Section \ref{sect::SurfContAngle}, we investigate surfaces making a constant angle with a parallel direction. The results of Section \ref{sect::SurfContAngle}, as well as those of Subsection \ref{subsect::SurfConstAngleM2xR}, generalize the findings of Dillen \emph{et al.} in $\mathbb{S}^2\times\mathbb{R}$ and $\mathbb{H}^2\times\mathbb{R}$ for surfaces making a constant angle with the vertical real direction \cite{dillen2007constantangleS2xR,dillen2009constantangleH2xR}. The main result of Section \ref{sect::SurfContAngle} is

\begin{abctheorem}\label{thrABC::SurfCteAngleAreRuledFlat}
    Let $\Sigma^2\subset M^3$ be a surface making a constant angle with a parallel transported direction. Then, $\Sigma^2$ is an extrinsically flat and ruled surface.
\end{abctheorem}

It is known that Euclidean slant helices are geodesics of surfaces of constant angle \cite{LucasBBMS2016}. In Section \ref{sect::SlantHelicesAndSurfConstAngle}, we investigate the possibility of generalizing this characterization to a generic ambient space. We obtain a relation between slant helices and rectifying surfaces in Subsection \ref{subsect::RectSurfAndConstantAngle}:

\begin{abctheorem}\label{thrABC::RectifyingSurfAsSurfCteAngle}
    If $\Sigma^2\subset M^3$ is a surface of constant angle, then there exists a slant helix $\gamma$ such that $\Sigma^2$ is parametrized as the rectifying surface $\mathcal{R}_{\gamma}$ of $\gamma$:
    \[
      \mathcal{R}_{\gamma}:\quad X(u,v) = \exp_{\gamma(u)}\Big(v\,\frac{\tau\,T+\kappa\,B}{\sqrt{\tau^2+\kappa^2}}(u)\Big).
    \]
   On the other hand, if $\gamma$ is a slant helix with axis $V$, then $\mathcal{R}_{\gamma}$ is a surface of constant angle if, and only if, the field obtained by parallel transporting $V$ along the rulings of $\mathcal{R}_{\gamma}$ lies in the kernel of the curvature operator $R(\frac{\partial X}{\partial u},\frac{\partial X}{\partial v})$.
\end{abctheorem}

The field obtained by parallel transport of $V$ along the rulings of $\mathcal{R}_{\gamma}$ is the natural candidate to be the axis of $\mathcal{R}_{\gamma}$. As a corollary of the second part of Theorem \ref{thrABC::RectifyingSurfAsSurfCteAngle}, we show that manifolds of constant non-vanishing curvature do not have surfaces of constant angle. In addition, every rectifying surface of a slant helix is a surface of constant angle if, and only if, $M^3$ is flat.

A version of Theorem \ref{thrABC::SurfCteAngleAreRuledFlat} is obtained in $\mathbb{S}^2\times\mathbb{R}$ and $\mathbb{H}^2\times\mathbb{R}$ by direct integration of the coordinates of the immersion from the constant angle condition \cite{dillen2007constantangleS2xR,dillen2009constantangleH2xR}. Our approach is based on the theory of extrinsically flat surfaces developed by the first and third authors of this work \cite{daSilvaSPJMS2022}. The Ansatz provided by Theorem \ref{thrABC::SurfCteAngleAreRuledFlat} allows us to establish in Subsect. \ref{subsect::SurfConstAngleM2xR} a classification of surfaces making a constant angle with the vertical real direction $\partial_t$ in a product manifold $M^2\times\mathbb{R}$.

\begin{abctheorem}
    A surface $\Sigma^2\subset M^2\times\mathbb{R}$ makes a constant angle $\theta$ with $\partial_t$ if, and only if, $\Sigma^2$ is locally parametrized as
    \[
        X(u,v) = \left(\exp^M_{\alpha(u)}\Big(v\cos\theta\,J\big(\alpha'(u)\big)\Big),v\sin\theta\right),
    \]
    where $\alpha:I\to M^2$ is a unit speed curve, $\exp^M$ is the exponential map in $M^2$, and $J$ is a $\frac{\pi}{2}$-rotation on the tangent planes of $M^2$.
\end{abctheorem}



\section{Curves of constant angle}
\label{sect::CurvesConstAngle}

In Euclidean space, the notion of constant direction is based on the fact that any two vectors at distinct locations can be compared by parallel transport. (Indeed, flat manifolds are characterized by the path-independence property of parallel transport.) However, the same can not be done for a generic 3D manifold. Thus, if we want to study curves or surfaces based on the angle they make with a ``fixed direction", an alternative is to demand parallel transport only along the curve or the surface. 

\begin{definition}
Let $M^3$ be a Riemannian manifold, $\gamma:I\to M^3$ be a regular curve, and $U$ be an arbitrary unit vector field along $\gamma$. The \emph{parallel angle} of $U$ is the angle between $U$ and a unit length parallel transported vector field $V$ along $\gamma$.  We say that the curve $\gamma$ is a
\begin{enumerate}[(a)]
    \item  \emph{parallel generalized helix} if its tangent has constant parallel angle with $V$;
    \item \emph{parallel slant helix} if its principal normal has constant parallel angle with $V$.
\end{enumerate}
In both cases, the vector field $V$ is said to be the \emph{axis} of the helix.
\end{definition}

The use of parallel transported vector fields gives rise to a quite flexible theory, as illustrated in Theorems \ref{thr::CharactParallelGenHelices} and \ref{thr::CharactParallelSlaHelices} in this section.

\subsection{Generalized parallel helices}

The classical Lancret Theorem asserts that generalized helices in Euclidean space are characterized by a constant ratio $\tau / \kappa$ of torsion to curvature. (See Theorem 15.1 of Ref. \cite{Kreyszig1991}.) We obtain a generalization of Lancret's theorem for parallel generalized helices on a generic Riemannian manifold. (See also Ref. \cite{Hayden1931helicesRiemannianManifolds}.)

\begin{theorem}\label{thr::CharactParallelGenHelices}
Let $\gamma: I\rightarrow M^{3}$ be a regular curve with curvature $\kappa>0$. Then, $\gamma$ is a parallel generalized helix if, and only if, the ratio $\tau / \kappa$ of torsion to curvature is constant. In addition, $\tau / \kappa=\cot(\theta)$ and the axis of the generalized helix is given by the vector field $V=\cos(\theta) T+\sin(\theta) B$, where $\{T,N,B\}$ is the Frenet frame of $\gamma$.
\end{theorem}
\begin{proof}
(An alternative proof is provided at the end of this section. This second proof will give a geometric interpretation for $\tau/\kappa$.)

Suppose $\gamma$ is a parallel generalized helix with axis $V$ and Frenet frame $\{T,N,B\}$. We may assume that $V$ has a unit length. Thus, $\nabla_TV=0$ and $\langle T, V \rangle = c_{1}$, where $c_{1}$ is a constant. Taking the derivative of $\langle T, V \rangle=c_1$ gives
\begin{equation}
    0 = \langle \nabla_{T}T, V \rangle + \langle V, \nabla_{T}V \rangle = \kappa \langle N, V\rangle.
\end{equation}
Since $\kappa>0$, we conclude that $\langle N, V\rangle=0$. Now, the derivative of $\langle B,V\rangle$ gives
\begin{equation}
    T\langle B, V\rangle = \langle \nabla_{T}B, V\rangle + \langle B, \nabla_{T}V\rangle = -\tau\langle N, V \rangle = 0,
\end{equation}
from which we conclude that $\langle B,V\rangle=c_2$ must be constant. Finally, the derivative of $\langle N,V\rangle=0$ gives
\begin{equation}
    0 = \langle \nabla_{T}N, V \rangle + \langle N, \nabla_{T}V\rangle = \langle -\kappa T + \tau B, V \rangle  = - c_{1}\kappa + c_2\tau. 
\end{equation}
In conclusion, $\tau/\kappa$ must be constant. In addition, since $V=c_1T+c_2B$ is a unit vector field, we may write $c_1=\cos\theta$ and $c_2=\sin\theta$, where $\theta$ is the angle between $T$ and $V$. Therefore, if $\gamma$ is a parallel generalized helix, then $\tau/\kappa=\cot\theta$. 

Conversely, let $\gamma$ be a curve with the property that $\tau = (\cot\theta) \kappa$, where $\theta$ is a constant. From Theorem 3.6 of Ref. \cite{CastrillonLopezAM2015}, we can always guarantee the existence of a unique curve with prescribed $\kappa,\tau$ and initial conditions $\{T(s_0),N(s_0),B(s_0)\}$. Now, define the following vector field $V$ along a curve $\gamma$ with curvature $\kappa$ and torsion $\tau = (\cot\theta) \kappa$
\begin{equation*}
    V = \cos\theta\,T+\sin\theta\,B.
\end{equation*}
Then, the derivative of $V$ is
\begin{equation*}
    \nabla_{T}V = \cos\theta\,\nabla_{T}T + \sin\theta\,\nabla_{T}B = (\kappa\cos\theta-\tau\sin\theta)\, N = 0.
\end{equation*}
Thus, $V$ is parallel transported along $\gamma$ and, in addition, 
\begin{equation*}
    \langle T, V\rangle = \cos\theta=\mbox{constant.}
\end{equation*}
In short, $\gamma$ is a parallel generalized helix making a constant parallel angle $\theta$ with $V$.
\end{proof}

\begin{example}[Helices in Product Manifolds]
    A characterization of helices in $\mathbb{S}^2\times\mathbb{R}$ with axis given by the real direction $\partial_t$ was provided by Nistor \cite{nistor2017constantangleS2xR}. (Note that $\partial_t$ is parallel transported along any curve in $\mathbb{S}^2\times\mathbb{R}$.) Her Theorem 1 provides an explicit parametrization of such curves. Thus, the theorem provides a coordinate-dependent characterization of parallel generalized helices with axis $\partial_t$. In addition, it is mentioned in Ref. \cite{nistor2017constantangleS2xR} that such curves are also generalized helices, meaning the ratio of torsion to curvature is constant. However, note that $\tau/\kappa=\mathrm{const.}$ does not only characterize generalized helices with axis $\partial_t$.
\end{example}

The Frenet Theorem fails in a generic manifold, i.e., two curves with the same curvature and torsion may not be related by a rigid motion of $M^3$. Indeed, the validity of the Frenet Theorem is a characteristic property of manifolds of constant curvature \cite{CastrillonLopezAM2015,CastrillonLopezDGA2014}. The existence of a curve with prescribed torsion and curvature still holds, but the initial conditions $\{T(s_0),N(s_0),B(s_0)\}$ must also be prescribed. Thus, once we fix an initial point $p=\gamma(s_0)$, a generalized helix depends on the prescription of the initial orientation of the curve $\{T(s_0),N(s_0),B(s_0)\}$ and the ratio $\tau/\kappa$. In other words, one real function and four real constants determine the family of parallel generalized helices. (The function can be either $\kappa$ or $\tau$, while two of the constants give $T(s_0)$, another gives $N(s_0)$ once we fix $T(s_0)$, and the remaining constant gives $\theta$.)

\subsection{Parallel slant helices}

Izumiya and Takeuchi \cite{IzumiyaTJM2004} introduced slant helices in Euclidean space. The constancy of a certain function of the curvature and torsion characterizes these curves. The class of slant helices can be generalized to a generic Riemannian manifold. We have the following characterization.

\begin{theorem}\label{thr::CharactParallelSlaHelices}
Let $\gamma: I\rightarrow M^{3}$ be a regular curve with curvature $\kappa>0$. Then, $\gamma$ is a parallel slant  helix if, and only if, 
\begin{eqnarray}\label{eq::SigmaSlantHelices}
 \sigma(s) = \left[\frac{\kappa^2}{(\kappa^2+\tau^2)^{\frac{3}{2}}}\frac{\rmd }{\rmd s}\left(\frac{\tau}{\kappa}\right)\right](s)
\end{eqnarray}
is a constant function. In addition, the axis of the slant helix is given by the vector field $V=\cos(\theta) N+\sin(\theta) D$, where $\sigma=\cot(\theta)$ and $D=(\tau T+\kappa B)/\sqrt{\tau^2+\kappa^2}$ is the normalized Darboux vector field of $\gamma$.
\end{theorem}
\begin{proof}
(An alternative proof is provided at the end of this section. This second proof will give a geometric interpretation for $\sigma$.)

Suppose $\gamma$ is a parallel slant  helix with axis $V$, i.e., $\langle V, V\rangle = 1$, $\nabla_TV=0$, and $\langle N, V \rangle = \cos\theta$, where $\theta$ is constant. Taking the derivative of $\langle N, V \rangle$ gives
\begin{equation*}
    0 = \langle \nabla_{T}N, V \rangle + \langle V, \nabla_{T}V \rangle = - \kappa \langle T, V \rangle + \tau \langle B, V \rangle=\langle - \kappa T+\tau B, V \rangle.
\end{equation*}
Then, we conclude that $V\in\mbox{span}\{-\kappa T+\tau B\}^{\perp}$ and, therefore, we can write
\begin{equation}\label{eq::ParallelSlantAxis}
 V = \cos\theta\, N+\sin\theta\,D, \quad D = \frac{\tau T+\kappa B}{\sqrt{\kappa^2+\tau^2}}. 
\end{equation}
Now, the derivative of $V$ in the direction of $T$ gives
\begin{eqnarray}
  \nabla_TV & = &  \cos\theta(-\kappa T + \tau B) - \sin\theta\frac{\kappa\kappa' + \tau\tau' }{(\kappa^{2}+\tau^{2})^{\frac{3}{2}}}(\tau T+\kappa B) +  \frac{\sin\theta}{\sqrt{\kappa^{2}+\tau^{2}}}(\tau' T + \kappa'B) \nonumber \\
            & = & \cos\theta(-\kappa T + \tau B) - \frac{\kappa^{2}}{(\kappa^{2}+\tau^{2})^{\frac{3}{2}}}\left(\frac{\tau}{\kappa}\right)'\sin\theta (-\kappa T + \tau B) \nonumber \\
            & = & (\cos\theta-\sigma\,\sin\theta)(-\kappa T+\tau B).\label{eq::NablaTofParSlantAxis}
\end{eqnarray}
This equation shows that $\nabla_TV=0$ if, and only if, $\sigma=\cot\theta$ is constant.

Conversely, let $\sigma$ be constant. Define $V$ as in Eq. \eqref{eq::ParallelSlantAxis}, where $\theta$ is the constant such that $\cot\theta=\sigma$. (From Theorem 3.6 of Ref. \cite{CastrillonLopezAM2015}, we can always guarantee the existence of a unique curve with initial conditions $\{T(s_0),N(s_0),B(s_0)\}$ and prescribed $\kappa,\tau$ satisfying $\sigma'(s)=0$. See also the discussion leading to Eq. \eqref{eq::TorsionOfSlantHelixGivenTheCurvature}.) Finally, similar computations as those leading to Eq. \eqref{eq::NablaTofParSlantAxis} give that $\nabla_TV=0$ and, consequently, $\gamma$ is a parallel slant helix with axis $V$.
\end{proof}

To characterize a slant helix, we need to prescribe the initial orientation of the curve, i.e., $\{T(s_0),N(s_0),B(s_0)\}$ at some point $p=\gamma(s_0)$, the constant $\theta$, and one real function, e.g., the curvature. Indeed, if a real function $\bar{\kappa}(u)>0$ is the curvature of a slant helix, then the torsion $\tau$ must be
\begin{equation}\label{eq::TorsionOfSlantHelixGivenTheCurvature}
 \tau = \pm \,\frac{\bar{\kappa}(\sigma K-c_0)}{\sqrt{1-(\sigma K-c_0)^2}}; \quad \vert c_0\vert<1, \quad K(u) = \int_{u_0}^u\bar{\kappa}(x)\mathrm{d}x. 
\end{equation}
Therefore, once we fix the initial point $p=\gamma(u_0)$, the members of the family of slant helices appear in pairs, the $\pm1$ in $\tau$. The family is parametrized by one real function and five real constants: one of the constants being the angle $\theta$, three others determining the orientation of the Frenet frame at $p$, and the remaining one determining the torsion at $u_0$, i.e., $c_0$. Note that a slant helix is guaranteed to exist on an interval $(u_m,u_M)$, where $u_m<u_0$ and $u_M>u_0$ are given by the condition $\sigma K(u^*)-c_0=\pm1$. (The values of $u_m,u_M$ may be infinite.)

We now investigate the geometric meaning of the constant $\sigma$. For slant helices in $\mathbb{E}^3$, Izumiya and Takeuchi \cite{IzumiyaTJM2004} indicated that characterization can be provided by noticing that $\gamma:I\to\mathbb{E}^3$ is a slant helix if, and only if, the \emph{normal indicatrix}, i.e., the principal normal seen as a curve $N:I\to\mathbb{S}^2$, is a small circle or, equivalently, if and only if $N:I\to\mathbb{S}^2$ has constant geodesic curvature. 

Given a unit speed curve $\alpha:I\to\mathbb{S}^2$, we can decompose its acceleration vector as $\ddot{\alpha}=(\ddot{\alpha}\cdot\alpha)\alpha+(\ddot{\alpha}\cdot\alpha\times\dot{\alpha})\alpha\times\dot{\alpha}$ and, therefore, its geodesic curvature is $\kappa_g=\alpha\cdot\dot{\alpha}\times\ddot{\alpha}$, where derivatives with respect to the arc-length $s_{\alpha}$ are denoted by a dot. Under a reparametrization $s_{\alpha}\mapsto t(s_{\alpha})$ of $\alpha$, we obtain $\kappa_g(t)=
[v^{-3}\alpha\cdot\alpha'\times\alpha''](t)$, where $v=\frac{\rmd s_{\alpha}}{\rmd t}=\Vert\frac{\rmd\alpha}{\rmd t}\Vert$ (notice $\frac{\rmd}{\rmd s_{\alpha}}=\frac{1}{v}\frac{\rmd}{\rmd t}$). Now, for the normal indicatrix, with arc-length parameter $s_{n}$, we have $v=\Vert\frac{\rmd N}{\rmd s}\Vert=\sqrt{\kappa^2+\tau^2}$. Thus, the geodesic curvature $\kappa_{g,n}$, under the change $s_n\mapsto s$, is
\begin{equation}\label{eq::GeodCurvNormalIndicatrix}
 \kappa_{g,n}(s)   = \frac{1}{(\kappa^2+\tau^2)^{\frac{3}{2}}}N\cdot\frac{\rmd N}{\rmd s}\times\frac{\rmd^2N}{\rmd s^2} = \frac{\tau'\kappa-\tau\kappa'}{(\kappa^2+\tau^2)^{\frac{3}{2}}}=\frac{\kappa^2}{(\kappa^2+\tau^2)^{\frac{3}{2}}}\frac{\rmd}{\rmd s}\left(\frac{\tau}{\kappa}\right). 
\end{equation}
Therefore, we can identify $\sigma=\kappa_{g,n}$. It follows that $\gamma$ is a slant helix if, and only if, $\sigma$ is constant. 

Finally, to relate $\sigma$ to the angle of the slant helix, we proceed as follows. A small circle  of $\mathbb{S}^2$ making an angle $\theta$ with its axis is a circle of radius $r=\sin\theta$ and, therefore, its geodesic curvature satisfies $\kappa_g^2=\kappa^2-\kappa_n^2=\frac{1}{\sin^2\theta}-1=\cot^2\theta$. Once again, we see that $\sigma$ is geometrically given by the cotangent of the constant angle $\theta\in[0,\frac{\pi}{2}]$. 

\begin{remark}\label{rem::AltenartiveCharGenHelicesUsingIndicatrices}
  A similar reasoning when applied to the \emph{tangent} and \emph{binormal indicatrices} $T:I\to\mathbb{S}^2$ and $B:I\to\mathbb{S}^2$ of $\alpha:I\to\mathbb{E}^3$, respectively, gives
\begin{equation}
    \kappa_{g,t}(s) = \frac{\tau}{\kappa}\mbox{ and }\kappa_{g,b}(s) = \frac{\kappa}{\vert\tau\vert}.
\end{equation}
Therefore, since $T:I\to\mathbb{S}^2$ has to be a small circle for a generalized helix, we conclude that a curve is a generalized helix if, and only if, $\kappa/\tau$ is constant. In addition, making a constant angle with the binormal is equivalent to making a constant angle with the unit tangent.  
\end{remark}

We can generalize the construction of tangent and normal indicatrices to a generic Riemannian manifold and provide another proof for characterizing parallel generalized and slant helices.

\begin{proof}[Alternative proof of Theorems \ref{thr::CharactParallelGenHelices} and \ref{thr::CharactParallelSlaHelices}.]
    We may construct the tangent, normal, and binormal indicatrices of curves in a Riemannian manifold $M^3$  with the help of parallel transport. (As indicated in Remark \ref{rem::AltenartiveCharGenHelicesUsingIndicatrices}, the tangent and normal indicatrices can be used to characterize generalized and slant helices, respectively.)
    
    Let us find the normal indicatrix. (The reasoning for the tangent indicatrix being analogous.) At every point of the curve, we have a vector $N$ on a unit sphere. However, the principal normal vectors at distinct points are not members of the same sphere. To circumvent this problem, first fix a point on the curve $p=\gamma(s_0)$ and compute the map $s\mapsto P_{s_0}(N(s))$ defined as the parallel transport of $N(s)$ along $\gamma$ from $\gamma(s)$ to $p$. Through this identification, we define the normal indicatrix as the curve $s\mapsto P_{s_0}(N(s))$ in $\mathbb{S}^2\subset T_{p}M^3$. Since parallel transport preserves angles, we can use the reasoning leading to Eq. \eqref{eq::GeodCurvNormalIndicatrix} to characterize parallel slant helices in $M^3$.
\end{proof}

\section{Helices and geodesics of cylinders}
\label{Sect::HelicesAndCylinders}

In Euclidean space, generalized helices are geodesics of right cylinders, and the cylinder's axis coincides with the curve's axis \cite{Kreyszig1991}. We may ask whether a similar characterization applies to helices in a generic manifold. 

Let $\beta:I\to M^3$ be a regular curve and $V$ a unit vector field parallel transported along $\beta$. The \emph{cylinder with basis $\beta$ and rulings parallel to} $V$ is defined as the surface
\begin{equation}
    \mathcal{C}_{\beta,V}(u,v) = \exp_{\beta(u)}(vV(u)).
\end{equation}
If $\beta'$ and $V$ are linearly independent, then $\mathcal{C}_{\beta,V}$ is a regular surface on a neighborhood of $\beta$. 

We have the following partial result concerning the relation between helices and cylinders.

\begin{proposition}\label{Prop::ParallelHelixIsGeodesiOfItsCylinder}
  If $\gamma:I\to M^3$ is a non-trivial generalized helix with axis $V$, i.e., $\gamma'$ and $V$ are linearly independent, then $\gamma$ is a geodesic of the cylinder $\mathcal{C}_{\gamma,V}$.
\end{proposition}
\begin{proof}  
    If we parametrize $\gamma$ by arc-length, $\langle\gamma',\gamma'\rangle=1$, then $\langle\nabla_{\gamma'}\gamma',\gamma'\rangle=0$. Taking the derivative of $\langle\gamma',V\rangle=\cos\theta$, $\theta$ constant, and using $\nabla_{\gamma'}V=0$ give
    \[
     0 = \langle \nabla_{\gamma'}\gamma',V\rangle+\langle\gamma',\nabla_{\gamma'}V\rangle=\langle\nabla_{\gamma'}\gamma',V\rangle.
    \]
    Thus, $\nabla_{\gamma'}\gamma'$ is a multiple of the unit normal of $\mathcal{C}_{\gamma,V}$, and, therefore, $\gamma$ is a geodesic.
\end{proof} 

Are the geodesics of cylinders also generalized helices? If yes, the natural candidate to play the role of an axis is the extension of $V$ to the remaining points of the cylinder. More precisely, let us extend $V$ by defining the vector field $V(u,v)=P_v{V}(u)$, where $P_vV$ is the parallel transport of $V$ along the geodesic $v\mapsto\exp_{\beta(u)}(vV(u))$ connecting $\mathcal{C}_{\beta,{V}}(u,0)$ to $\mathcal{C}_{\beta,{V}}(u,v)$.


\begin{lemma}
    Let $\gamma:I\to\mathcal{C}_{\beta,V}$ be a geodesic. If $V$ is parallel transported along $\gamma$, i.e., if the extension of $V$ to the remaining points of the cylinder satisfies $\nabla_{\gamma'}V=0$, then $\gamma$ is a generalized helix.
\end{lemma}
\begin{proof}
    First, note that $\langle V,V\rangle$ is constant on $\mathcal{C}_{\beta,V}$. Indeed, $\Vert V(u_1,v_1)\Vert=\Vert V(u_1,0)\Vert$ (parallel transport along the rulings), $\Vert V(u_1,0)\Vert=\Vert V(u_2,0)\Vert$ (transport along $\beta$), and $\Vert V(u_2,0)\Vert=\Vert V(u_2,v_2)\Vert$ (transport along the rulings).
    
    Now, since $\gamma$ is a geodesic, $\nabla_{\gamma'}\gamma'$ is orthogonal to $V$: $\langle \nabla_{\gamma'}\gamma',V\rangle=0$. (Assume $\gamma$ is parametrized by arc length.) Taking the derivative of $\langle\gamma',V\rangle$ gives
    \[
     \gamma'\langle\gamma',V\rangle = \langle\nabla_{\gamma'}\gamma',V\rangle+\langle\gamma',\nabla_{\gamma'}V\rangle = 0+0 =0.
    \]
    Consequently, $\gamma'$ makes a constant angle with $V$. Thus, $\gamma$ is a helix with axis $V$.
\end{proof}

The previous lemma suggests it is enough to prove that $V$ is parallel transported along any curve of $\mathcal{C}_{\beta,V}$. However, demanding the validity of such a property may impose limitations on the geometry of the cylinder. Analogously, requiring the validity of this property on \emph{any} cylinder may also impose limitations on the geometry of the ambient space.

\begin{theorem}\label{thr::WhenCylindricalGeodesicsAreHelices}
    Let $\mathcal{C}_{\beta,V}$ be a cylinder in $M^3$. Then, $V$ is parallel transported along any curve of $\mathcal{C}_{\beta,V}$ if, and only if, $\mathcal{C}_{\beta,V}$ is simultaneously intrinsically and extrinsically flat. In addition, a Riemannian manifold $M^3$ is flat if, and only if, the axis of any cylinder $\mathcal{C}_{\beta,V}\subset M^3$ is parallel transported along any curve of the cylinder.
\end{theorem}
\begin{proof}  
    Let $V(u,v)$, the extension of $V(u)$ to the remaining points of the cylinder, be parallel transported along any curve of $\mathcal{C}_{\beta,V}$, i.e., for every $X\in\Gamma(\mathcal{C}_{\beta,V})$, we have $\nabla_XV=0$. Since $\nabla_VV=0$, verifying the property for $X=\partial\mathcal{C}/\partial u\equiv U$ is enough. Let $\nu$ be the unit normal of $\mathcal{C}_{\beta,V}$ and $\nabla^{\mathcal{C}}$ be the induced Levi-Civita connection, then
    \[
     \nabla_UV = \nabla^{\mathcal{C}}_UV + h_{12}\,\nu,
    \]
    where $h_{ij}$ denotes the coefficients of the second fundamental form of $\mathcal{C}_{\beta,V}$. Therefore,
    \[
     \nabla_UV = 0 \Leftrightarrow \nabla^{\mathcal{C}}_UV = 0 \quad \mbox{and} \quad h_{12} = 0.
    \]
    Since $\nabla_VV=0$, we have $h_{22}=0$, from which follows that $K_{ext}=0\Leftrightarrow h_{12}=0$. Thus, $\mathcal{C}_{\beta,V}$ is extrinsically flat. In addition, $\nabla^{\mathcal{C}}_XV = 0$ for any $X\in\Gamma(\mathcal{C}_{\beta,V})$ means $\mathcal{C}_{\beta,V}$ is also intrinsically flat. In short, $V$ is parallel transported along any curve of $\mathcal{C}_{\beta,V}$ if, and only if, $\mathcal{C}_{\beta,V}$ is both intrinsically and extrinsically flat.

   In addition, if the axis of \emph{any} cylinder $\mathcal{C}_{\beta,V}\subset M^3$ is parallel transported along any curve of the cylinder, then the Gauss' Theorem $K_{int}-K_{ext}=K_{sec}$ implies all sectional curvatures vanish ($V$ and $\beta$ are arbitrary). Thus, we finally deduce that $M^3$ must be flat. 
   
   Conversely, if $M^3$ is flat, it is locally isometric to the Euclidean space $\mathbb{E}^3$. The fact that $V$ is parallel transported along any cylindrical curve in Euclidean space is a well-known property.
\end{proof}

\begin{example}
     Vertical cylinders of $\mathbb{S}^2\times\mathbb{R}$ are intrinsically and extrinsically flat and, consequently, its geodesics must be parallel generalized helices with vertical axis.   Theorem 1 of Ref. \cite{nistor2017constantangleS2xR} characterizes parallel generalized helices in $\mathbb{S}^2\times\mathbb{R}$ with axis $\partial_t$; they are parametrized as 
    \[
     \gamma(s) = (\sin\theta\int^s\cos\alpha(\zeta)\rmd\zeta, \sin\theta\int^s\frac{\sin\alpha(\zeta)}{\cos\varphi(\zeta)}\rmd\zeta,s\cos\theta+t_0),
    \]
    where $\theta$ and $t_0$ are constant, $\alpha$ is a smooth function, and $\varphi$ is the first coordinate of $\gamma$. (Nistor adopts geographic coordinates $(\varphi,\psi)$ for $\mathbb{S}^2$.) The third coordinate of $\gamma$ is linear. Then, the principal normal $N$ of $\gamma$ is tangent to $\mathbb{S}^2$; therefore, $\gamma$ is a geodesic of the cylinder.
\end{example}

\section{Surfaces of constant parallel angle}
\label{sect::SurfContAngle}

Let $\Sigma^2$ be a regular surface of a smooth Riemannian manifold $M^3$. Let us denote by ${\nabla}$ the Levi-Civita connection of $M^3$ and by $\snabla$ the induced connection of $\Sigma^2$. If we denote by $h$ the second fundamental form of $\Sigma^2$, then
\begin{equation}
    X,Y\in\mathfrak{X}(\Sigma^2),\,\nabla_XY = \snabla_XY +h(X,Y). 
\end{equation}
The shape operator of $\Sigma^2$ is defined as $A(X)=-\nabla_X\nu$, where $X$ is tangent, and $\nu$ is the unit normal to $\Sigma^2$. The curvature operator $R$ of $M^3$ is denoted by
\begin{equation}
    R(X,Y)Z = \nabla_Y\nabla_XZ - \nabla_X\nabla_YZ + \nabla_{[X,Y]}Z.
\end{equation}

\begin{definition}
 A surface $\Sigma^2\subset M^3$ with unit normal $\nu$ is a \emph{surface of constant parallel angle} if there exists a unit vector field $V$ parallel transported along \emph{any} curve in $\Sigma^2$ such that $V$ and $\nu$ make a constant angle. We shall refer to $V$ as the \emph{axis} of $\Sigma^2$.
\end{definition}

If  $V$ is the axis of a surface of constant parallel angle $\theta$, then we may decompose it as
\begin{equation}\label{eq::DecompositionOfVtheConstantDirection}
    V = T + \cos\theta\, \nu,\quad T = \sin\theta \,\mathbf{e}_1, 
\end{equation}
where $\mathbf{e}_1$ is a unit tangent vector field and $\theta$ is constant. Let $\mathbf{e}_2$ be the tangent unit vector field orthogonal to $\mathbf{e}_1$ such that $\{\mathbf{e}_1,\bfe_2,\nu\}$ is a positive basis.

The following result generalizes Proposition 1 of Ref. \cite{dillen2007constantangleS2xR} and Proposition 2.1 of Ref. \cite{dillen2009constantangleH2xR} by Dillen \emph{et al.}.

\begin{proposition}\label{prop::PropertiesOfTangentProjOfParallelDirection}
Let $X$ be a tangent vector field to $\Sigma^2$ and $V$ a parallel direction along $\Sigma^2$ that makes an angle $\theta$ with the unit normal $\nu$, where $\theta$ is not necessarily constant. Then, decomposing $V$ as in Eq. \eqref{eq::DecompositionOfVtheConstantDirection}, we obtain
\begin{equation}
    \snabla_XT = \cos\theta\,AX\quad\mbox{and}\quad X(\cos\theta) = -\langle AT,X\rangle.
\end{equation}
Moreover, if $\theta$ is constant, then $T$ is a principal direction of $\Sigma^2$ with vanishing principal curvature.
\end{proposition}
\begin{proof}
Taking the derivative of $V$ in the direction of $X$ and decomposing it in its tangent and normal parts, we have
\begin{eqnarray}
 \nabla_XV & = & \nabla_XT+\nabla_X(\cos\theta\,\nu) = \snabla_XT+h(X,T)+X(\cos\theta)\,\nu+\cos\theta\,\nabla_X\nu \nonumber\\
            & = & [\snabla_XT-\cos\theta\,AX]+[h(X,T)+X(\cos\theta)\nu].
\end{eqnarray}
By hypothesis, $\nabla_XV=0$ and, therefore, the tangent and normal parts of $\nabla_XV$ obtained above give $\snabla_XT=\cos\theta AX$ and $X(\cos\theta)\nu=-h(X,T)$, respectively. From the properties of the second fundamental form, it follows that
\begin{equation}
 X(\cos\theta) = \langle X(\cos\theta)\nu,\nu\rangle = -\langle h(X,T),\nu\rangle = -\langle AT,X\rangle.  
\end{equation}
Finally, if $\theta$ is constant, then $\langle AT,X\rangle=0$ for all tangent $X$, which implies $AT=0$.
\end{proof}

\begin{remark}
    The projection of the axis of a surface of constant angle is a principal direction of the surface. Thus, they have the \emph{principal direction property.} (See the discussion in the final paragraph of the Concluding remarks section \ref{sect::Conclusion}.)
\end{remark}

From Proposition \ref{prop::PropertiesOfTangentProjOfParallelDirection}, it follows that the shape operator of a surface of constant parallel angle can be written in the basis $\{\bfe_1,\bfe_2\}$ as
\begin{equation}\label{eq::ShapeOpOfConstParalAngleSurf}
    A = \left(
        \begin{array}{lr}
            0 & \,\,0  \\
            0 & \,\,\lambda \\
        \end{array}
        \right),
\end{equation}
where the function $\lambda$ is the principal curvature associated with $\bfe_2$. Thus, the second fundamental form of $\Sigma^2$ is written as
\begin{equation}\label{eq::2ndFFConstAngleSurf}
    h(\bfe_1,\bfe_1)=0,\quad h(\bfe_2,\bfe_1)=h(\bfe_1,\bfe_2)=0,\quad h(\bfe_2,\bfe_2)=\lambda\,\nu.
\end{equation}

The expressions for the shape operator and second fundamental form prove that surfaces of constant parallel angle must be extrinsically flat. Indeed, from Eq. \eqref{eq::2ndFFConstAngleSurf}, it follows that $h_{1i}=0$ and, therefore, $K_{ext}=0$:

\begin{theorem}\label{thr::ConstParAngleSurfAreExtFlat}
Let $\Sigma^2\subset M^3$ be a surface of constant parallel angle. Then, $\Sigma^2$ is extrinsically flat.
\end{theorem}

Since surfaces of constant parallel angle are extrinsically flat, it is natural to ask whether they are also ruled, as would be the case in spaces of constant curvature \cite{daSilvaSPJMS2022}. We will prove that this is indeed the case. But, first, we need to understand better the intrinsic geometry of surfaces of constant angle.

\begin{proposition}\label{prop::LeviCivitaConnectionOfConstParalAngleSurf}
Let $\Sigma^2$ be a surface making a constant angle $\theta$ with a parallel direction $V$. Then, the Levi-Civita connection $\snabla$ of $\Sigma^2$ is given by
\begin{equation}
    \snabla_{\bfe_1}\bfe_1 = 0,   \quad
        \snabla_{\bfe_1}\bfe_2 = 0, \quad \snabla_{\bfe_2}\bfe_1 = \lambda\cot\theta\,\bfe_2,  \quad \snabla_{\bfe_2}\bfe_2 = -\lambda\cot\theta\,\bfe_2. 
\end{equation}
In addition, $[\bfe_1,\bfe_2]=-\lambda\cot\theta\,\bfe_2$.
\end{proposition}
\begin{proof}
From Prop. \ref{prop::PropertiesOfTangentProjOfParallelDirection}, we can write
\begin{equation}
    A\bfe_i = \tan\theta\,\snabla_{\bfe_i}\bfe_1.
\end{equation}
From Eq. \eqref{eq::ShapeOpOfConstParalAngleSurf}, we obtain $\snabla_{\bfe_1}\bfe_1=0$ and $\snabla_{\bfe_2}\bfe_1=\lambda\cot\theta\,\bfe_2$.

Since $\bfe_2$ is a unit vector, we conclude that $\snabla_{\bfe_i}\bfe_2$ is parallel to $\bfe_1$. From the orthogonality of $\bfe_1$ and $\bfe_2$, it follows that
\begin{equation}
    \langle\bfe_1,\snabla_{\bfe_i}\bfe_2\rangle = - \langle\bfe_2,\snabla_{\bfe_i}\bfe_1\rangle = \left\{
    \begin{array}{cr}
        0 & \mbox{ if }i=1  \\
        -\lambda\cot\theta & \mbox{ if }i=2  \\ 
    \end{array}
    \right..
\end{equation}
Finally, the Lie bracket $[\bfe_1,\bfe_2]$ expression follows the torsion-free property of the Levi-Civita connection.
\end{proof}

It is known that surfaces in $\mathbb{S}^2\times\mathbb{R}$ and $\mathbb{H}^2\times\mathbb{R}$ making a constant angle with the vertical direction $\partial_t$ are extrinsically flat and ruled. Such characterizations appear as Theorem 2 of Ref. \cite{dillen2007constantangleS2xR} (in $\mathbb{S}^2\times\mathbb{R}$) and as Theorem 3.2 of Ref. \cite{dillen2009constantangleH2xR} (in $\mathbb{H}^2\times\mathbb{R}$). The proofs of Dillen \emph{et al.} are coordinate-dependent and employ the constant angle condition to integrate the coordinate functions of the immersion. At the end of the integration, one can finally conclude that the surfaces are also ruled. (The property of being ruled is not explicitly emphasized in Refs.  \cite{dillen2007constantangleS2xR,dillen2009constantangleH2xR}.) 

Now, we are ready to establish the main theorem of this section, which shows that being extrinsically flat and ruled is a generic property of surfaces of constant parallel angle. Our approach will exploit the theory of ruled surfaces developed in Ref. \cite{daSilvaSPJMS2022}. For an extrinsically flat surface to be ruled, its asymptotic curves must be geodesics of the ambient space, and this condition can be translated in terms of the vanishing of a certain component of the curvature tensor of the ambient space. Namely, Proposition 5 of Ref. \cite{daSilvaSPJMS2022} establishes that an extrinsically flat surface $\Sigma^2\subset M^3$ is ruled if, and only if, $R_{trrn}=0$, where $n$ refers 
to the direction normal to $\Sigma^2$, $r$ refers to the direction of the rulings, and $t$ refers to the remaining tangent direction.

\begin{theorem}\label{thr::ConstParAngleSurfAreRuled}
Let $\Sigma^2\subset M^3$ be a surface of constant parallel angle $\theta$ and axis $V$. Then, $\Sigma^2$ is an extrinsically flat and ruled surface. In addition, the rulings are tangent to the axis' projection over $\Sigma^2$.
\end{theorem}
\begin{proof}
From Theorem \ref{thr::ConstParAngleSurfAreExtFlat}, we know that $\Sigma^2$ is extrinsically flat. The natural candidates to be the rulings are the integral lines of the projection of $V$ over the tangent planes of $\Sigma^2$. From Prop. 5 of Ref. \cite{daSilvaSPJMS2022}, $\Sigma^2$ is ruled if, and only if, the component $R_{2113}=0$, where the indexes $i=1,2,$ and $3$ refer to the directions of $\mathbf{e}_1$ (asymptotic direction of $\Sigma^2$), $\mathbf{e}_2$, and $\nu$. 

Using Prop. \ref{prop::LeviCivitaConnectionOfConstParalAngleSurf} in this section and that $h(\mathbf{e}_1,\mathbf{e}_j)=0$ in \eqref{eq::2ndFFConstAngleSurf}, we can compute
\begin{eqnarray}
  R(\bfe_2,\bfe_1)\bfe_1 & = & \nabla_{\bfe_1}\nabla_{\bfe_2}\bfe_1-\nabla_{\bfe_2}\nabla_{\bfe_1}\bfe_1+\nabla_{[\bfe_2,\bfe_1]}\bfe_1 \nonumber \\
  & = & \nabla_{\bfe_1}\snabla_{\bfe_2}\bfe_1-\nabla_{\bfe_2}\snabla_{\bfe_1}\bfe_1+\lambda\cot\theta\snabla_{\bfe_2}\bfe_1 \nonumber \\
  & = & [(\bfe_1\cdot\snabla\lambda)\cot\theta + (\lambda\cot\theta)^2]\,\bfe_2.
\end{eqnarray}
Thus, we finally obtain that
\begin{equation}
    R_{2113} = \langle R(\bfe_2,\bfe_1)\bfe_1,\nu\rangle = 0.
\end{equation}
\end{proof}

\begin{corollary}\label{cor::KintSurfacesCTEangle}
    Let $\Sigma^2\subset M^3$ be a surface of constant parallel angle $\theta$. Then, the intrinsic curvature of $\Sigma^2$ is given by
    \[
      K_{int} = -(\lambda_{,1}+\lambda^2\cot\theta)\cot\theta,
    \]
    where $\lambda_{,1}=\bfe_1\cdot\nabla\lambda$ and $\lambda$ is defined in Eq. \eqref{eq::ShapeOpOfConstParalAngleSurf}.
\end{corollary}

\section{Slant helices and geodesics of surfaces of constant parallel angle}
\label{sect::SlantHelicesAndSurfConstAngle}

To our knowledge, the concept of slant helices was first introduced by Izumiya and Takeuchi in Euclidean space in the 2000s \cite{IzumiyaTJM2004}. A decade later, Lucas and Ortega-Yag\"ues characterized Euclidean slant helices as geodesics of surfaces of constant angle \cite{LucasBBMS2016}.

In this section, we show that the surface of constant angle containing a Euclidean slant helix $\gamma$ is the rectifying developable surface $\mathcal{R}_{\gamma}$ of $\gamma$. The rectifying developable surface is the envelope of the rectifying planes, and it is parametrized as the ruled surface $(u,v)\mapsto\gamma(u)+v(\tau T+\kappa B)(u)$ \cite{struik}. We define rectifying surfaces for curves on a generic manifold and show that if a surface of constant angle exists, it must be a rectifying surface of a slant helix. We show that space forms of nonzero curvature do not admit surfaces of constant angle. In addition, we prove that flat manifolds are characterized by the property that the rectifying surface of any slant helix is a surface of constant parallel angle. Finally, we characterize surfaces of constant angle on Riemannian products with the real line.

\subsection{Rectifying surfaces}

The rectifying developable surface $\mathcal{R}$ of an Euclidean curve $\gamma$ is shown to be a flat ruled surface. These properties justify the terminology for $\mathcal{R}$. Indeed, since it has vanishing Gaussian curvature, $\mathcal{R}$ can be developed into a plane. Moreover, since $\gamma$ is a geodesic of its rectifying surface, this process then locally maps $\gamma$ to a straight line, i.e., $\gamma$ was rectified.

Let $\gamma:I\to\mathbb{E}^3$ be a regular smooth curve with Frenet frame $\{T,N,B\}$. The plane spanned by $T$ and $B$ is known as the \emph{rectifying plane}. We can define a surface $\mathcal{R}$, the \emph{rectifying developable}, as the envelope of the rectifying planes. Thus, $\mathcal{R}$ is implicitly given by
\begin{equation}\label{eq::DefRectifyingDevInE3}
    X\in\mathcal{R}:[X-\gamma(u)]\cdot N(u)=0, \quad u\in I.
\end{equation}
Taking the derivative with respect to $u$,
\begin{equation}
   0 = -T\cdot N+(X-\gamma)\cdot(-\kappa T+\tau B)=(X-\gamma)\cdot(-\kappa T+\tau B).
\end{equation}
Therefore, as in the intersection of two families of planes, $\mathcal{R}$ must be a ruled surface whose rulings are parallel to 
\begin{equation}
    Z = N\times(-\kappa T+\tau B)=\tau T+\kappa B,
\end{equation}
and we can parameterize $\mathcal{R}$ as
\begin{equation}
    X(u,v)=\gamma(u)+vZ(u).
\end{equation}
The vector field $\tau T+\kappa B$ is often called the \emph{Darboux vector field} of $\gamma$.

\begin{remark}
   A similar construction of rectifying developable surfaces as envelopes of a 1-parameter family of planes can be done in a space form. Indeed, seeing the 3-sphere $\mathbb{S}^3(r)$ and hyperbolic space $\mathbb{H}^3(r)$ as submanifolds of $\mathbb{E}^4$ and $\mathbb{E}_1^4$, respectively, totally geodesic surfaces play the role of planes. In these models, totally geodesic surfaces are given by intersections with planes of $\mathbb{R}^4$ passing through the origin \cite{Spivak1979v4}.  
\end{remark}


\begin{definition}
    Let $M^3$ be a smooth Riemannian manifold. The \emph{rectifying surface} of a regular smooth curve $\gamma:I\to M^3$ is defined as the ruled surface
    \begin{equation}
      \mathcal{R}_{\gamma}:(u,v){\mapsto} X(u,v)=\exp_{\gamma(u)}(vD(u)), \quad D = \frac{\tau T+\kappa B}{\sqrt{\kappa^2+\tau^2}}.
    \end{equation}
\end{definition} 

We will show that rectifying surfaces in space forms are extrinsically flat; therefore, we shall also refer to them as rectifying developable surfaces. However, we can not guarantee that a rectifying surface is extrinsically flat in a generic manifold. Thus, we omit the word ``developable" from the definition.

In a space form, the extrinsic Gaussian curvature of a rectifying surface is proportional to $-\langle\gamma'\times D,\nabla_{\gamma'}D\rangle^2$ \cite{daSilvaSPJMS2022}. (Here, $D$ is a vector, not a derivative, as in Ref. \cite{daSilvaSPJMS2022}. In addition, we can define a vector product in $\mathbb{S}^3(r)$ and $\mathbb{H}^3(r)$ using the ternary product in $\mathbb{E}^4$ and $\mathbb{E}_1^4$ \cite{daSilvaSPJMS2022}, respectively.) A straightforward calculation shows that $\mathcal{R}_{\gamma}$ is always extrinsically flat in a space form. Moreover, along $\gamma$, the unit normal $\nu_{\mathcal{R}}$ of $\mathcal{R}$ is parallel to
\[
 T\times D = -\frac{\tau}{\sqrt{\kappa^2+\tau^2}}N,
\]
which implies $\gamma$ is a geodesic of its rectifying developable surface. (The fact that $\gamma$ is a geodesic of its rectifying surface is also valid in a generic manifold; see Theorem \ref{thr::CharacRectifyingSurf}.)


\begin{remark}
    In Refs. \cite{LucasJMAA2015,LucasMJM2016},  Lucas and Ortega-Yag\"ues show that $\gamma$ is a rectifying curve if, and only if, its rectifying developable is a cone. Thus, generalizing a result by Izumiya and Takeuchi in Euclidean space \cite{IzumiyaTJM2004}. However, the authors of Refs. \cite{LucasJMAA2015,LucasMJM2016} do not show that their rectifying developable surfaces are extrinsically flat.
\end{remark}

We shall prove that the rectifying surface is the only ruled surface deserving the adjective ``rectifying". (The theorem below generalizes Prop. 4.1 of Ref. \cite{IzumiyaTJM2004} in Euclidean space.)

\begin{theorem}\label{thr::CharacRectifyingSurf}
    Let $\gamma:I\to M^3$ be a regular smooth curve. The rectifying surface $\mathcal{R}_{\gamma}\subset M^3$ of $\gamma$ is the only regular ruled surface containing $\gamma$ and satisfying the properties
    \begin{enumerate}
        \item $\gamma$ is a geodesic of $\mathcal{R}_{\gamma}$, and
        \item $\mathcal{R}_{\gamma}$ is extrinsically flat along $\gamma$.
    \end{enumerate}
\end{theorem}
\begin{proof}
   Let $Z$ be the director field of the rulings of a ruled surface $\Sigma^2$ satisfying conditions 1 and 2. Condition 1 implies that $Z=\lambda T+\mu B$, for some smooth real functions $\lambda,\mu$. The extrinsic curvature of a ruled surface is proportional to $-vol( X_u,X_v,\nabla_{X_u}X_v)^2$, see, e.g., Eq. (5) of Ref. \cite{CastrillonLopezArXiv2023}. Thus, condition 2 is valid if, and only if, $vol( T,Z,\nabla_{T}Z)=0$. Using that
   \[
   \nabla_TZ = \lambda'T+(\lambda \kappa-\mu\tau)N+\mu'B,
   \]
   we obtain
   \[
    0 =  vol( T,Z,\nabla_{T}Z) = \mu(\lambda \kappa-\mu\tau)\,vol(T,B,N) = \mu(\mu\tau-\lambda \kappa).
   \]
   Since $\mu\not=0$ (we want regular surfaces), we deduce that $\mu\tau-\lambda \kappa = 0$ and, therefore, $Z$ is a multiple of $\tau T+\kappa B$. In short, 1 and 2 implies $\Sigma^2$ is the rectifying surface $\mathcal{R}_{\gamma}$.
\end{proof}

\subsection{Rectifying surfaces and constant parallel angle}
\label{subsect::RectSurfAndConstantAngle}

Let $\gamma$ be a geodesic of a surface $\Sigma^2$ of constant parallel angle. It follows that $\gamma$  must be a slant helix. Indeed, being a geodesic implies that the principal normal $N$ of $\gamma$ is a multiple of the surface unit normal $\nu$. Since $\nu$ makes a constant parallel angle with a fixed direction, so does $N$. Moreover, surfaces of constant parallel angle are necessarily extrinsically flat and ruled, Theorem \ref{thr::ConstParAngleSurfAreRuled}. Therefore, our Theorem \ref{thr::CharacRectifyingSurf} implies that surfaces of constant parallel angle must be rectifying surfaces of slant helices. 

Now, we provide an alternative and shorter proof of the characterization of Euclidean slant helix as geodesics of surfaces of constant angle surface \cite{LucasBBMS2016}.

\begin{theorem}\label{thr::CharactConstAnglSurfInE3}
Let $\gamma:I\to\mathbb{E}^3$ be a regular curve. Then, $\gamma$ is a slant helix if, and only if, it is the geodesic of a surface of constant angle $\Sigma^2$. In addition, $\Sigma^2$ is unique, and it is the rectifying developable surface of $\gamma$ 
\end{theorem}
\begin{proof}
At the beginning of this subsection, we showed that geodesics of surfaces of constant angle must be slant helices. (Note that this implication does not depend on the ambient space.)

Conversely, let $\gamma$ be a slant helix and let $\mathcal{R}_{\gamma}$ be its rectifying developable surface. Note $\gamma$ is a geodesic in $\mathcal{R}_{\gamma}$. Since (i) the normal $\nu$ of $\mathcal{R}_{\gamma}$ along $\gamma$ is given by the principal normal of $\gamma$ and (ii) $\nu$ is stationary along the rulings, then $\mathcal{R}_{\gamma}$ makes a constant angle with a fixed direction. 

Finally, the uniqueness of $\mathcal{R}_{\gamma}$ as the only surface of constant angle containing the slant helix $\gamma$ as a geodesic follows from Theorem \ref{thr::CharacRectifyingSurf}, as discussed at the beginning of this subsection. 
\end{proof}

We already know that geodesics of surfaces of constant parallel angle are parallel slant helices. We shall now investigate the converse. Namely, is the rectifying surface $\mathcal{R}_{\gamma}\subset M^3$ of a parallel slant helix $\gamma$ always a surface of constant parallel angle? The study of generalized helices as geodesics of cylinders suggests that the answer may depend on properties of $M^3$ (see Theorem \ref{thr::WhenCylindricalGeodesicsAreHelices}).

Let $\gamma:I\to M^3$ be a parallel slant helix with Frenet frame $\{T,N,B\}$, curvature $\kappa$, and torsion $\tau$. Then, its axis is given by $V=\cos\theta N+\sin \theta D$, where $D=(\tau T+\kappa B)/\sqrt{\kappa^2+\tau^2}$ denotes the unit Darboux vector field of $\gamma$ (Theorem \ref{thr::CharactParallelSlaHelices}). Let  $\mathcal{R}_{\gamma}:X(u,v)=\exp_{\gamma(u)}(vD(u))$ be the rectifying surface of $\gamma$. Then, $\partial_v=\partial X/\partial v$ is parallel transported along the rulings. 

If $\mathcal{R}_{\gamma}$ is a surface of constant angle, then  Theorem \ref{thr::ConstParAngleSurfAreRuled} implies $\mathcal{R}_{\gamma}$ is extrinsically flat and, therefore, its unit normal $\nu$ is parallel transported along the rulings. Thus, if $\mathcal{R}_{\gamma}$ is a surface of constant angle, its axis must be given by 
\begin{equation}\label{eq::AxisOfRectifyingSurfOfSlantHelix}
    V = \sin\theta\,\partial_v - \cos\theta\, \nu.
\end{equation}
Note the minus sign ensures that along $\gamma$, the basis $\{\partial_u,\partial_v,\nu\}$ has the same orientation as the Frenet frame of $\gamma$.

The rectifying surface is a surface of constant parallel angle provided that the axis $V$ is parallel transported along any curve on the surface. By construction, $\nabla_{\partial_v}V=0$. Thus, it remains to check whether $\nabla_{\partial_u}V=0$.

If $\mathcal{R}_{\gamma}$ is extrinsically flat, it follows that $h_{12}=0$ and, therefore,
\[
 \langle \nu,\nabla_{\partial_u}V\rangle = \sin\theta\langle \nu,\nabla_{\partial_u}\partial_v\rangle - \cos\theta\langle \nu,\nabla_{\partial_u}\nu\rangle = \sin\theta\,h_{12} = 0
\]
and
\[
 \langle \partial_v,\nabla_{\partial_u}V\rangle = \sin\theta\langle \partial_v,\nabla_{\partial_u}\partial_v\rangle - \cos\theta\langle \partial_v,\nabla_{\partial_u}\nu\rangle = \cos\theta\,h_{12} = 0.
\]
These equalities hold even if $\nabla_{\partial_u}V\not=0$, as their validity relies on $K_{ext}=0$ only.

Thus, in a generic manifold $M^3$,
\[
 \nabla_{\partial_u}V = 0 \Leftrightarrow \langle \partial_u,\nabla_{\partial_u}V\rangle = 0. 
\]
We have
\begin{equation}\label{eq::NablaDelUV}
    \langle \partial_u,\nabla_{\partial_u}V\rangle = \sin\theta\langle \partial_u,\nabla_{\partial_u}\partial_v\rangle - \cos\theta\langle \partial_u,\nabla_{\partial_u}\nu\rangle = \frac{g_{11,2}}{2}\sin\theta+h_{11}\cos\theta.
\end{equation}
In short, the equation $\frac{1}{2}g_{11,2}\sin\theta+h_{11}\cos\theta=0$ represents the necessary and sufficient condition for an extrinsically flat rectifying surface of a given slant helix be a surface of constant parallel angle.

\begin{proposition}\label{prop::NoSurfConstAngleInSpaceForm}
    There exists no surface in $\mathbb{S}^3(r)$ and $\mathbb{H}^3(r)$ making a constant parallel angle with a fixed direction.
\end{proposition}
\begin{proof}
    We will provide an alternative proof for this proposition at the end of this section, which will be useful in characterizing the manifolds with the property that every rectifying surface of a slant helix is a surface of constant parallel angle.
    
    Let us do the proof for $\mathbb{S}^3(r)$. The computations for $M=\mathbb{H}^3(r)$ are analogous. First, note that 
    \[
     (\frac{\kappa}{\sqrt{\kappa^2+\tau^2}})' = -\tau\sigma \quad \mbox{ and } \quad (\frac{\tau}{\sqrt{\kappa^2+\tau^2}})' = \kappa\sigma\,,
    \]
    where $\sigma$ is defined as in Eq. \eqref{eq::SigmaSlantHelices}. From Eq. (9) of Ref. \cite{daSilvaSPJMS2022},
    \begin{eqnarray}
        g_{11} & = & c_v^2 + 2 r s_vc_v \langle T,\nabla_TD\rangle + r^2s_v^2\langle \nabla_TD,\nabla_TD\rangle + s_v^2\langle T,D\rangle^2 \nonumber\\
        & = & c_v^2 + 2 r \kappa\sigma s_vc_v + r^2\sigma^2(\kappa^2+\tau^2)s_v^2 + s_v^2\frac{\tau^2}{\omega^2},
    \end{eqnarray}
    where we adopted the shorthand notation $s_v = \sin(v/r)$, $c_v = \cos(v/r)$, and $\omega = \sqrt{\kappa^2+\tau^2}$. Taking the derivative with respect to $v$ gives
    \begin{equation}
        g_{11,2} = \left(r\omega^2\sigma^2-\frac{\kappa^2}{r\omega^2}\right)s_{2v} + 2 \kappa \sigma c_{2v}. 
    \end{equation}
    In addition, using that $g_{22}=1$ and $g_{12}=\langle T,D\rangle = \tau/\omega$, we can write the determinant of the metric as
    \begin{equation}
        g = (\frac{\kappa}{\omega}c_v+r\omega\sigma s_v)^2 \Rightarrow \sqrt{g} = (\frac{\kappa}{\omega}c_v+r\omega\sigma s_v).
    \end{equation}
    Now, the coefficient $h_{11}$ can be obtained from Prop. 1 of Ref. \cite{daSilvaSPJMS2022}. (Attention to the differences in notation. Here, $D$ is the unit Darboux vector field, not a covariant derivative). After slightly long but straightforward calculations, we deduce that 
    \begin{eqnarray}
        h_{11} & = & \frac{c_v(\nabla_TT,T,D)+rs_v[(\nabla_TT,\nabla_TD,D)
     +(\nabla_T^2D,T,D)]}{\sqrt{g}} \nonumber \\ 
     & & +\frac{r^2s^2_v(\nabla_T^2D,\nabla_TD,D)}{c_v\sqrt{g}} \nonumber \\ 
     & = & -\frac{\omega}{c_v\sqrt{g}}\left(\frac{\kappa}{\omega}c_v+r\omega\sigma s_v\right)^2 = -\frac{1}{c_v}\left(\kappa c_v+r\omega^2\sigma s_v\right).
    \end{eqnarray}
    
    We finally obtain
    \begin{equation}\label{eq::g112SinMinush11CosSphere}
        \frac{g_{11,2}}{2}\sin\theta+h_{11}\cos\theta = -\frac{1}{r\omega^2}\tan\frac{v}{r}\left(\kappa\cos\frac{v}{r}+r\omega^2\sigma\sin\frac{v}{r}\right)^2,
    \end{equation}
    which vanishes only along $\gamma$, i.e., $v=0$.  Therefore, we conclude that a rectifying surface of a slant helix in a space form of positive or negative curvature is never a surface of constant parallel angle. 
\end{proof}

\begin{remark}
    The reader can easily verify that, in Euclidean space, the rectifying developable of any slant helix satisfies $\frac{1}{2}g_{11,2}\sin\theta+h_{11}\cos\theta=0$. Alternatively, we could take a shortcut and consider the limit $r\to \infty$ in Eq. \eqref{eq::g112SinMinush11CosSphere}. Thus, we obtain
    \begin{equation}
        (\frac{g_{11,2}}{2}\sin\theta+h_{11}\cos\theta)\vert_{\mathbb{E}^3} = \lim_{r\to \infty}\left[ -\frac{1}{r\omega^2}\tan\frac{v}{r}\left(\kappa\cos\frac{v}{r}+r\omega^2\sigma\sin\frac{v}{r}\right)^2\right] = 0.
    \end{equation}
    Therefore, surfaces of constant angle in Euclidean space do exist as expected.
\end{remark}

What can be said about a manifold with the property that every rectifying surface of a slant helix is also a surface of constant parallel angle? To answer this question, we shall adopt an approach alternative to that leading to the vanishing of the right-hand side of Eq. \eqref{eq::NablaDelUV}. Although Eq. \eqref{eq::NablaDelUV} leads to a necessary and sufficient condition for the rectifying surface of a slant helix to be of constant angle, it is written in terms of properties of the rectifying surface, thus making the role played by the ambient space less transparent.

The Eq. \eqref{eq::NablaDelUV} was obtained by noticing that a rectifying surface $\mathcal{R}_{\gamma}$ is a surface of constant angle if, and only if, $V$ is parallel transported along any curve of the surface. Since $\nabla_{\partial_v}V=0$ by construction, see Eq. \eqref{eq::AxisOfRectifyingSurfOfSlantHelix}, it suffices to check whether $\nabla_{\partial_u}V=0$. It turns out that $\nabla_{\partial_u}V=0$ is equivalent to $\nabla_{\partial_v}\nabla_{\partial_u}V=0$. Indeed, if we can show that $\nabla_{\partial_v}\nabla_{\partial_u}V=0$, then $\nabla_{\partial_u}V$ is parallel transported along the rulings of the rectifying surface. However, $\nabla_{\partial_u}V=0$ along the slant helix $\gamma$, which implies $\nabla_{\partial_u}V=0$ on the remaining points of $\mathcal{R}_{\gamma}$. Therefore, the rectifying surface of a slant helix makes a constant angle with a parallel direction if, and only if, $\nabla_{\partial_v}\nabla_{\partial_u}V=0$.

If $R$ denotes the Riemann curvature tensor of $M^3$, we can write
\[
 R(\partial_u,\partial_v)V =  \nabla_{\partial_v}\nabla_{\partial_u}V - \nabla_{\partial_u}\nabla_{\partial_v}V+\nabla_{[\partial_u,\partial_v]}V = \nabla_{\partial_v}\nabla_{\partial_u}V.
\]
Consequently, the rectifying surface of a slant helix makes a constant angle with a parallel direction if, and only if, $R(\partial_u,\partial_v)V=0$. Using the expression for $V$ in Eq. \eqref{eq::AxisOfRectifyingSurfOfSlantHelix}, the condition $R(\partial_u,\partial_v)V=0$ becomes
\begin{equation}
    0 = R(\partial_u,\partial_v)V = \sin\theta\, R(\partial_u,\partial_v)\partial_v - \cos\theta\, R(\partial_u,\partial_v)\,\nu.
\end{equation}
Taking the inner product with $\partial_u$ gives the following necessary condition for $\mathcal{R}_{\gamma}$ to be a surface of constant angle
\begin{equation}
    \sin\theta\, R_{uvvu} - \cos\theta\, R_{uvnu} = 0 \Rightarrow \cot\theta\, R_{uvun} = R_{uvuv},
\end{equation}
where the indices $i=u,v,n$ refer to the directions $\partial_u$, $\partial_v$, and $\nu$, respectively.

\begin{proof}[Alternative proof of Proposition \ref{prop::NoSurfConstAngleInSpaceForm}]
    Let $\gamma$ be a non-trivial slant helix, $\sigma=\cot\theta\not=0$, and let $\mathcal{R}_{\gamma}$ be its rectifying surface. In a space form, any component of the curvature tensor with three distinct indices must vanish. Then, $R_{uvun} = 0$. On the other hand,  $R_{ijij}\not=0$ and, therefore, we cannot have $ R_{uvuv} = \cot\theta\,R_{uvun}$. Thus, there exists no surface in $\mathbb{S}^3(r)$ or $\mathbb{H}^3(r)$ of constant parallel angle.
\end{proof}

We can now characterize those manifolds with the property that \emph{every} rectifying surface of a slant helix is a surface of constant angle.

\begin{theorem}
    Let $M^3$ be a smooth Riemannian manifold. Then, $M^3$ is a flat manifold if, and only if, every rectifying surface of a slant helix is a surface of constant parallel angle.
\end{theorem}
\begin{proof}
    From Theorem \ref{thr::CharactConstAnglSurfInE3}, we know that in flat manifolds, any rectifying surface of a slant helix is a surface of constant parallel angle.
    
    Conversely, fix a point $p\in M^3$ and a plane $\Pi\subset T_pM^3$. Consider a non-trivial slant helix $\gamma_{\theta}$ making a constant angle $\theta$ with a fixed direction and such that $\Pi=\mathrm{span}\{T,D\}$ at $p=\gamma_{\theta}(u_0)$. (See discussion leading to Eq. \eqref{eq::TorsionOfSlantHelixGivenTheCurvature}.) We must then have
    \[
     \sin\theta\, R_{uvuv}(p) - \cos\theta\, R_{uvun}(p) = 0. 
    \]
    Note that $\theta$ can take any value in $(0,\frac{\pi}{2})$. Therefore, $R_{uvuv}(p)=0$. From the arbitrariness of $p$ and $\Pi$, we conclude that every sectional curvature of $M^3$ must vanish. Thus, $M^3$ is a flat manifold.
\end{proof}

\subsection{Surfaces of constant parallel angle in product manifolds}
\label{subsect::SurfConstAngleM2xR}

On a manifold with a globally defined parallel transported vector field $V$, such as the vertical direction $\partial_t$ of a product manifold $M^2\times\mathbb{R}$, the condition $R(\partial_u,\partial_u)V=0$ is trivially satisfied. Therefore, every rectifying surface of a slant helix $\gamma:I\to M^2\times\mathbb{R}$ with axis $\partial_t$ is a surface of constant parallel angle.

The metric of $M^2\times\mathbb{R}$ can be written as
\[
 g = g^M + \rmd t^2,
\]
where $g^M$ denotes the metric of $M^2$ and $\rmd t^2$ the standard metric of $\mathbb{R}$. If we denote by  $R^M$ the curvature tensor of $M^2$, we can write the curvature tensor of $M^2\times\mathbb{R}$ as
\begin{equation}
    R(X,Y,Z,W) = R^M(X_M,Y_M,Z_M,W_M),
\end{equation}
where $X_M$ denotes the component of $X$ tangent to $M^2$: $X=X_M+\langle X,\partial_t\rangle\partial_t$.

Let $\Sigma^2\subset M^2\times\mathbb{R}$ be a surface of constant parallel angle $\theta$ and axis $V$. Let $\mathbf{e}_1$ and $\mathbf{e}_2$ be defined as in Eq. \eqref{eq::DecompositionOfVtheConstantDirection}. We can then write
\begin{equation}
    \mathbf{e}_1 = \mathbf{E}_1 + \sin\theta\,\partial_t, \quad \hat{\mathbf{E}}_1 = \frac{\mathbf{E}_1}{\Vert\mathbf{E}_1\Vert}, \quad \mbox{and} \quad \mathbf{e}_2 = \hat{\mathbf{E}}_2,
\end{equation}
where $\hat{\mathbf{E}}_1$ and $\hat{\mathbf{E}}_2$ are tangent to $M^2$. Note that $\mathbf{E}_1=\cos\theta\,\hat{\mathbf{E}}_1$. Indeed,
\[
 g^M({\mathbf{E}}_1,{\mathbf{E}}_1) = \langle\mathbf{e}_1 - \sin\theta\,\partial_t,\mathbf{e}_1 - \sin\theta\,\partial_t\rangle = 1 - 2\sin\theta\langle\mathbf{e}_1,\partial_t\rangle + \sin^2\theta = \cos^2\theta.
\]

\begin{proposition}\label{Prop::KintSurfCteAngleInM2xR}
    Let $\Sigma^2\subset M^2\times\mathbb{R}$ be a surface of constant parallel angle $\theta$. If we denote by $K^M$ the (intrinsic) Gaussian curvature of $M^2$, then the intrinsic Gaussian curvature $K_{int}$ of $\Sigma^2$ is given by
    \begin{equation}
        K_{int} = K^M\cos^2\theta.
    \end{equation}
\end{proposition}
\begin{proof}
    The Gaussian curvature of $M^2$ can be written as $K^M=R^M(\hat{\mathbf{E}}_1,\hat{\mathbf{E}}_2,\hat{\mathbf{E}}_1,\hat{\mathbf{E}}_2)$. Now, using the fact that $\Sigma^2$ is extrinsically flat in $M^2\times\mathbb{R}$, we obtain
    \begin{eqnarray*}
     R^{\Sigma}(\mathbf{e}_1,\mathbf{e}_2,\mathbf{e}_1,\mathbf{e}_2) & = & R(\mathbf{e}_1,\mathbf{e}_2,\mathbf{e}_1,\mathbf{e}_2) = R^M(\mathbf{E}_1,\mathbf{e}_2,\mathbf{E}_1,\mathbf{e}_2) \nonumber \\
    & = & R^M(\cos\theta\,\hat{\mathbf{E}}_1,\hat{\mathbf{E}}_2,\cos\theta\,\hat{\mathbf{E}}_1,\hat{\mathbf{E}}_2) \nonumber \\
    & = & \cos^2\theta\,R^M(\hat{\mathbf{E}}_1,\hat{\mathbf{E}}_2,\hat{\mathbf{E}}_1,\hat{\mathbf{E}}_2) \nonumber \\
    & = & K^M\cos^2\theta.
\end{eqnarray*}
\end{proof}

Together with Corollary \ref{cor::KintSurfacesCTEangle}, the quantity $\lambda$ in Eq. \eqref{eq::ShapeOpOfConstParalAngleSurf} determining the shape operator of a surface $\Sigma^2$ of constant parallel angle in a product manifold is a solution of the differential equation
\begin{equation}\label{eq::RiccatiEqForLambda}
    \lambda_{,1}+\lambda^2\cot\theta = - \frac{1}{2}K^M\sin2\theta,
\end{equation}
where $\lambda_{,1}=\bfe_1\cdot\nabla\lambda$ and $\mathbf{e}_1$ is the unit length vector obtained from the projection of the axis $V$ over $\Sigma^2$. Therefore, $\lambda$ is the solution of a Riccati equation. The values of $\lambda$ are then determined by prescribing $\lambda$ along a curve of $\Sigma^2$.

\begin{remark}
    The expression for the intrinsic curvature given by Prop. \ref{Prop::KintSurfCteAngleInM2xR} generalizes the results by Dillen \emph{et al.} in $\mathbb{S}^2\times\mathbb{R}$ and $\mathbb{H}^2\times\mathbb{R}$. (See Prop. 2 of Ref. \cite{dillen2007constantangleS2xR} and Prop. 2.3 of Ref. \cite{dillen2009constantangleH2xR}.) The Riccati equation we obtained for $\lambda$ can be compared with Eq. (24) of Ref. \cite{dillen2007constantangleS2xR} and Eq. (18) of Ref. \cite{dillen2009constantangleH2xR}. The left-hand side of Eq. \eqref{eq::RiccatiEqForLambda} is obtained by an abstract manipulation of the curvature tensor of a surface of constant parallel angle and of the ambient space, while the right-hand side of Eq. \eqref{eq::RiccatiEqForLambda} relies on the explicit knowledge of the curvature tensor.
\end{remark}

Dillen \emph{et al.} studied surfaces making a constant angle with the vertical direction $\partial_t$ in $\mathbb{S}^2\times\mathbb{R}$ \cite{dillen2007constantangleS2xR} and $\mathbb{H}^2\times\mathbb{R}$ \cite{dillen2009constantangleH2xR}. They obtained an explicit parametrization of constant angle surfaces: (please note the difference in notation for the parameters $u$ and $v$)
\[
 F(u,v) = (\cos(v\cos\theta)f(u)+\sin(v\cos\theta)f(u)\times f'(u),v\,\sin\theta) \in \mathbb{S}^2\times\mathbb{R}
\]
and
\[
 F(u,v) = (\cosh(v\cos\theta)f(u)+\sinh(v\cos\theta)f(u)\boxtimes f'(u),v\,\sin\theta)\in \mathbb{H}^2\times\mathbb{R},
\]
where $\boxtimes$ denotes the vector product in the Lorentzian space $\mathbb{E}_1^3$. Note that the horizontal coordinates of these immersions provide a parametrization of the horizontal manifold by semi-geodesic coordinates. The operators $X\mapsto f\times X$ and $X\mapsto f\boxtimes X$ can be seen as $\frac{\pi}{2}$-rotation in $\mathbb{S}^2$ and $\mathbb{H}^2$, respectively. These observations imply that only the vertical coordinate determines whether a ruled surface in $M^2\times\mathbb{R}$ makes a constant angle with $\partial_t$.

\begin{lemma}\label{lemma::ParameOfRuledSurfInM2xRBySemiGeodCoordOfM2}
    Let $\Sigma^2\subset M^2\times\mathbb{R}$ be any regular ruled surface. If $\exp^M$ denotes the exponential map of $M^2$, then we can locally parametrize $\Sigma^2$ as
    \[
      X(u,v) = (x(u,v),t(u,v))=\Bigg(\exp^M_{\alpha(u)}\Big(v\,J\big(\alpha'(u)\big)\Big),a(u)v+b(u)\Bigg),
    \]
    where $\alpha:I\to M^2$ is a unit speed curve, $a,b:I\to\mathbb{R}$ are smooth functions, and $J$ is a $\frac{\pi}{2}$-rotation on the tangent planes of $M^2$.
\end{lemma}
\begin{proof}
    A curve of $M^2\times\mathbb{R}$ is a geodesic if, and only if, its horizontal and vertical projections are a geodesic of $M^2$ and a geodesic of $\mathbb{R}$, respectively. Therefore, every ruled surface in $M^2\times\mathbb{R}$ induces a foliation of $M^2$ by geodesics of $M^2$. Choosing a curve $\alpha:I\to M^2$ orthogonal to the geodesics of $M^2$, we obtain a parametrization of $M^2$ as $x(u,v)=\exp^M_{\alpha(u)}(v\,J\alpha'(u))$. In addition, we may take the lift of $\alpha$ to $\Sigma^2$ as the new generating curve. Finally, any geodesic of $\mathbb{R}$ is parametrized as $av+b$.
\end{proof}

We already know that any surface in $M^2\times\mathbb{R}$ making a constant angle with the vertical direction must be a ruled surface (Theorem \ref{thr::ConstParAngleSurfAreRuled}). Using the Ansatz provided by Lemma \ref{lemma::ParameOfRuledSurfInM2xRBySemiGeodCoordOfM2}, we can easily characterize surfaces of constant angle in $M^2\times\mathbb{R}$ by describing the function $a(u)$ and $b(u)$ only. 

\begin{theorem}\label{thr::RepresentationCteAngSurfInM2xR}
    A surface $\Sigma^2\subset M^2\times\mathbb{R}$ makes a constant parallel angle $\theta$ with the real direction if, and only if, up to isometries of $M^2\times\mathbb{R}$, $\Sigma^2$ is locally parametrized as
    \begin{equation}\label{eq::RepresentationCteAngSurfInM2xR}
        X(u,v) = (x(u,v),t(u,v)) = \Bigg(\exp^M_{\alpha(u)}\Big(v\cos\theta\,J\big(\alpha'(u)\big)\Big),v\sin\theta\Bigg),
    \end{equation}
    where $\alpha:I\to M^2$ is a unit speed curve, $\exp^M$ is the exponential map in $M^2$, and $J$ is a $\frac{\pi}{2}$-rotation on the tangent planes of $M^2$.
\end{theorem}
\begin{proof}
Let $\Sigma^2$ be parametrized as in Eq. \eqref{eq::RepresentationCteAngSurfInM2xR}. Its tangent vectors are $X_u=(x_u,0)$ and $X_v=(x_v,\sin\theta)$. (Note that $\langle x_v,x_v\rangle=\cos^2\theta$.) Then, the unit normal of $\Sigma^2$ is
\[ 
 \nu = - \tan\theta(x_v,0) + \cos\theta\,\partial_t.
 \]
Finally, $\langle\nu,\partial_t\rangle=\cos\theta$ and, therefore, $\Sigma^2$ is a surface of constant parallel angle.

Conversely, let $\Sigma^2$ be a surface of constant parallel angle. Let us parametrize $\Sigma^2$ as in Lemma \ref{lemma::ParameOfRuledSurfInM2xRBySemiGeodCoordOfM2}:
\[
 X(u,v)=(x(u,v),t(u,v))=(\exp^M_{\alpha(u)}(vJ\alpha'(u)),a(u)v+b(u)).
\]
Note that $x(u,v)$ is an orthogonal parametrization of $M^2$ with $\Vert x_v\Vert=1$.

The director field of the rulings, $X_v=(x_v,a)$, makes a constant angle $\frac{\pi}{2}-\theta$ with $\partial_t$. (This follows from Eq. \eqref{eq::DecompositionOfVtheConstantDirection} and the fact that the rulings are tangent to $\mathbf{e}_1$.) Then,
\[
 a = \langle X_v,\partial_t\rangle = \Vert X_v\Vert\Vert\partial_t\Vert\cos(\frac{\pi}{2}-\theta)=\sqrt{1+a^2}\,\sin\theta \, \Rightarrow \, a = \tan\theta.
\]

Along the generating curve, i.e., $v=0$, the unit normal of $\Sigma^2$ is given by
\begin{equation}
    \nu(u,0) = \frac{1}{\sqrt{1+\tan^2\theta+b'^2}}\,\Big(-b'(x_u,0)-\tan\theta\,(x_v,0)+\partial_t\Big).
\end{equation}
Using that $\Sigma^2$ makes a constant angle $\theta$ with $\nu$ gives
\[
 \cos\theta = \langle \nu,\partial_t\rangle = \frac{1}{\sqrt{1+\tan^2\theta+b'^2}} \quad \Rightarrow \quad b(u) = b_1u+b_0.
\]
We can set $b_0=0$ by performing a vertical translation. Substituting $b=b_1u$ in the constant angle condition implies 
\[
 1 = \cos^2\theta(1+\tan^2\theta+b_1^2) = \cos^2\theta + \sin^2\theta + b_1^2\cos^2\theta = 1 + b_1^2\cos^2\theta \Rightarrow b_1 = 0.
\]
Therefore, performing the reparametrization $v\mapsto v\cos\theta$, we finally conclude that
\[
  X(u,v) = (\exp^M_{\alpha}(v\cos\theta\,J\alpha'),v\sin\theta).
\]
\end{proof}

Theorem \ref{thr::RepresentationCteAngSurfInM2xR} can be seen as a particular instance of Theorem 5 of Ref. \cite{MTVdV20}. See the final paragraph of the Concluding remarks section.

\section{Concluding remarks}
\label{sect::Conclusion}

We studied curves and surfaces in a generic Riemannian manifold that make a constant angle with a parallel transported field. The classification of Euclidean generalized and slant helices regarding their curvature and torsion was proved valid for \emph{any} ambient space. However, this scenario changed dramatically when we studied cylinders and surfaces of constant angle and the relation between their geodesics and helices. Flatness emerged as a recurring theme in this context: demanding that all geodesics of \emph{any} cylinder and \emph{any} surface of constant angle are respectively generalized and slant helices with the same axes as their surfaces are only possible on flat manifolds. 

Quite surprisingly, the idea of using parallel transported directions as axes of surfaces in isotropic manifolds such as the sphere and hyperbolic space proved unfruitful; there exists no surface of constant parallel angle and no intrinsic cylinder whose geodesics are all generalized helices in space forms with positive or negative curvature. The problem is that demanding parallel transport along two independent directions is too restrictive. Thus, one idea to circumvent the null results in non-flat space forms would be to relax the condition of parallel transport along any curve on the surface. Instead, we could, for example, introduce a weak notion of constant angle surface by demanding that the axis is parallel transported along one family of parametric curves only. It is unclear how much of the theory presented in Sections \ref{sect::SurfContAngle} and \ref{sect::SlantHelicesAndSurfConstAngle} would remain valid for surfaces of constant angle in this weak sense. 

Concerning helices as geodesics of cylinders or surfaces of constant angle, we only considered the possibility that the helices' axes coincide with the axis of the surface. Thus, it remains to investigate whether these geodesics could be helices but with a distinct axis. For example, if $\gamma$ is a generalized helix which is also a geodesic of a cylinder $\mathcal{C}_{\beta,V}$, then the axis of $\gamma$ could be obtained by rotating on the tangent planes of the cylinder\footnote{If $\gamma$ is a geodesic of $\mathcal{C}_{\beta,V}$, then its principal normal is a unit normal for $\mathcal{C}_{\beta,V}$ along $\gamma$. If, in addition, $\gamma$ is a generalized helix, its axis is orthogonal to its principal normal; see Theorem \ref{thr::CharactParallelGenHelices}. Therefore, the axis of a cylindrical geodesic must be tangent to the cylinder.} the vector field $\bar{V}$ defined as the extension of $V$ by parallel transport along the rulings of $\mathcal{C}_{\beta,V}$. The possibility of rotating $\bar{V}$ in this way may depend on properties of the ambient space or, otherwise, it may be the case that a cylindrical geodesic $\gamma$ is also a generalized helix if, and only if, its axis coincides with the cylinder's axis.

As discussed in the Introduction, there is another way of defining a ``fixed direction"; we may consider Killing vector fields. Intuitively, defining a submanifold of constant angle with respect to a parallel transported direction only demands moving a single vector along the submanifold. On the other hand, if we use Killing fields, the intuition is that of isometrically deforming the entire submanifold in the direction of a vector so that we can perform the angular measurement. Therefore, one may naively expect a theory of constant angle based on parallel transport to be more flexible than one based on Killing fields. Nonetheless, using Killing fields to define surfaces of constant angle proved more fruitful in at least one of the contexts where our theory based on parallel transport failed, namely, manifolds of constant positive curvature. Indeed, in the 1990s, Barros proposed defining generalized helices as curves that make a constant angle with a Killing vector field of constant length \cite{BarrosPAMS1997}. Analogously, one can define slant helices with respect to Killing vector fields of constant length \cite{LucasJKMS2017}. This alternative theory in the sphere $\mathbb{S}^3$ provides beautiful theorems for helices seen as 3d curves or geodesics of surfaces of constant angle and certain cylinders. Unfortunately, there are still limitations, as the same success is not verified in the hyperbolic space $\mathbb{H}^3$. For example, only trivial helices exist in $\mathbb{H}^3$ \cite{BarrosPAMS1997,LucasJKMS2017}, i.e., curves with constant curvature and torsion. Such negative results seem to contradict intuition; after all, $\mathbb{H}^3$ is richer in isometries than the sphere and Euclidean space\footnote{To illustrate this point of view, consider two-dimensional space forms. The three independent Killing fields of $\mathbb{S}^2$ are essentially the same: a rotation around an axis. The three independent Killing fields of $\mathbb{E}^2$ are essentially two: rotation around a point and translation along a direction. However, the three independent Killing fields of $\mathbb{H}^2$ are indeed three; hyperbolic, parabolic, and elliptic rotations generate orbits with curvature $0<\kappa<1$, $\kappa=1$, and $\kappa>1$, respectively.}. This failure may be because the Killing fields serving as axes must have constant length. It certainly makes sense to demand that a vector field that plays the role of a fixed direction must have a constant length. However, if the goal is to hierarchize submanifolds based on their ``simplicity", one can and should weaken the notion of ``fixed direction" if necessary and then proceed to arrange the geometric objects of interest on the next level of the hierarchy. To complete the theory of curves and surfaces of constant angle, one must admit that the axes can be Killing fields whose length is not necessarily constant. This is presently under investigation by the authors and will be the subject of a follow-up
work.

Finally, yet another way to generalize the study of constant angle is by selecting one of the properties that follow from the constant angle condition and then studying submanifolds with such a property. For example, from Proposition \ref{prop::PropertiesOfTangentProjOfParallelDirection}, the tangent projection of the axis $V$ of a surface $\Sigma^2$ of constant angle is a principal direction. We say that $\Sigma^2$ has the \emph{principal direction property}, see Ref. \cite{MTVdV20} and references therein. A study of submanifolds with the principal direction property has been recently done by Manfio \emph{et al.} on space forms, product spaces $\mathbb{S}^n\times\mathbb{R}$ and $\mathbb{H}^n\times\mathbb{R}$, and warped products $I\times_{\rho}\mathbb{Q}_{\varepsilon}^n$ where $I\subseteq\mathbb{R}$ and $\mathbb{Q}_{\varepsilon}^n$ is a simply connected space form of constant sectional curvature $\varepsilon\in\{-1,0,1\}$ \cite{MTVdV20}. The study of the principal direction property has some advantages, as the vector field whose projection is a principal direction does not need to be of unit length. In addition, the property is invariant under conformal transformations, implying, for example, that one can translate the results about surfaces in $\mathbb{S}^2\times\mathbb{R}$ making a constant angle with $\partial_t$ into results about surfaces in $\mathbb{E}^3$ making a constant angle with the radial direction \cite{MTVdV20}. In the context of the principal direction property, our Theorem \ref{thr::RepresentationCteAngSurfInM2xR} can be seen as a particular instance of Theorem 5 of Ref. \cite{MTVdV20}.

\backmatter

\bmhead{Acknowledgements} 
Luiz da Silva acknowledges the support provided by the Mor\'a Miriam Rozen Gerber fellowship for Brazilian postdocs and also the Faculty of Physics Postdoctoral Excellence Fellowship during his training at the Weizmann Institute of Science. In addition, he thanks the hospitality of the Department of Mathematics at the Federal Rural University of Pernambuco and all the delicious coffee during two visits in 2023 when significant parts of this manuscript were written. Finally, the authors would like to thank Joeri van der Veken for bringing to our attention the idea of extending the study of constant angle using the principal direction property.

\bmhead{Conflict of Interest Statement} 
All authors certify that they have no affiliations with or involvement in any organization or entity with any interest in the subject matter or materials discussed in this manuscript.


\end{document}